\documentclass[11pt]{article}

\usepackage[margin=1.2in]{geometry}

\usepackage[T1]{fontenc}
\usepackage[utf8]{inputenc}
\usepackage[USenglish, UKenglish, english]{babel}
\usepackage{amssymb}
\usepackage{amsthm}
\usepackage{amsmath}
\usepackage[mathscr]{euscript}
\usepackage{enumitem}
\usepackage{graphicx}
\usepackage{MnSymbol}
 \let\mathscr\relax
\usepackage[scr]{rsfso}
\usepackage{tikz-cd}
\usepackage{tikz}
\usetikzlibrary{matrix,arrows,decorations.pathmorphing}
\usepackage{hyperref}
\usepackage{marginnote}
\usetikzlibrary{arrows.meta}

\usepackage[raggedright]{titlesec}

\theoremstyle{definition}
\newtheorem{defin}{Definition}[section]
\theoremstyle{definition}
\newtheorem{ex}[defin]{Example}
\theoremstyle{plain}
\newtheorem{theo}[defin]{Theorem}
\theoremstyle{plain}
\newtheorem{prop}[defin]{Proposition}
\theoremstyle{plain}
\newtheorem{lem}[defin]{Lemma}
\theoremstyle{plain}
\newtheorem{cor}[defin]{Corollary}
\theoremstyle{definition}

\theoremstyle{definition}

\theoremstyle{definition}
\newtheorem{pb}[defin]{Problem}
\theoremstyle{plain}

\theoremstyle{definition}

\theoremstyle{definition}

\theoremstyle{definition}
\newtheorem*{defin*}{Definition}
\theoremstyle{definition}
\newtheorem*{ex*}{Example}
\theoremstyle{plain}
\newtheorem*{theo*}{Theorem}
\theoremstyle{plain}
\newtheorem*{prop*}{Proposition}
\theoremstyle{plain}
\newtheorem*{lem*}{Lemma}
\theoremstyle{plain}
\newtheorem*{cor*}{Corollary}
\theoremstyle{definition}
\newtheorem*{rmk*}{Remark}
\theoremstyle{definition}
\newtheorem*{exe*}{Exercise}

\theoremstyle{plain}
\newtheorem{conjA}{Conjecture}[section]

\theoremstyle{plain}
\newtheorem{theoA}[conjA]{Theorem}

\numberwithin{equation}{section}

\usepackage{libertine}

\usepackage{sectsty}

\allsectionsfont{\centering}

\usepackage{parskip}

\makeatletter
\def\thm@space@setup{%
  \thm@preskip=\parskip \thm@postskip=0pt
}
\makeatother

\setlist[enumerate]{label=(\roman*)}

\def\irr{{\rm Irr}}

\def\ker{{\rm Ker}}

\def\syl{{\rm Syl}}

\def\SL{{\rm SL}}

\def\GL{{\rm GL}}

\def\n{{\mathbf{N}}}

\def\z{{\mathbf{Z}}}

\def\O{{\mathbf{O}}}

\def\N{{\mathcal{N}}} 

\def\Nm{{\mathcal{N}_m}} 

\def\F{{\mathbf{F}}} 

\makeatletter
\newcommand{\uset}[3][0ex]{%
  \mathrel{\mathop{#3}\limits_{
    \vbox to#1{\kern-7\ex@
    \hbox{$\scriptstyle#2$}\vss}}}}
\makeatother

\newcommand{\wh}[1]{\widehat{#1}}

\makeatletter
\def\blfootnote{\gdef\@thefnmark{}\@footnotetext}
\makeatother

\title{{\huge\bf Monomial characters of finite solvable groups}\\
\author{\Large Damiano Rossi}
\date{}
\blfootnote{\emph{$2010$ Mathematical Subject Classification:} $20$C$15$, $20$D$10$, $20$D$15$.
\\
\emph{Key words and phrases:} Monomial characters, non-vanishing elements, solvable groups.
\\
The content of this paper is part of the author's master thesis. This work is supported by the EPSRC grant EP/T$004592/1$. I would like to Silvio Dolfi for his guidance during this project and for introducing me to representation theory of finite groups. 
}}

\begin{document}

\renewcommand{\theconjA}{\Alph{conjA}}

\renewcommand{\thetheoA}{\Alph{theoA}}

\selectlanguage{english}

\maketitle

\begin{abstract}
We give new evidences to the fact that the structure of a solvable group can be controlled by irreducible monomial characters. In particular we inspect the role of monomial characters in Isaacs-Navarro-Wolf's conjecture and in Gluck's conjecture.
\end{abstract}

\section{Introduction}

Let $G$ be a finite group and denote by $\irr(G)$ the set of irreducible complex characters of $G$. As is well known the structure of $G$ is strongly influenced by properties of the set $\irr(G)$. In \cite{Lu-PanI}, \cite{Lu-PanII}, \cite{Che-LewI}, \cite{Che-LewII} and \cite{Che-Yan}, the authors showed that certain aspects of the structure of a solvable group $G$ can be controlled by only considering monomial characters. Following this idea, we propose refinements of two well-known open conjectures in character theory of finite groups introduced by Isaacs--Navarro--Wolf and Gluck respectively.

Recall that an element $g\in G$ is \emph{non-vanishing} if $\chi(g)\neq 0$ for every $\chi\in\irr(G)$. We denote by $\N(G)$ the set of non-vanishing elements of $G$. Let $\irr_m(G)$ be the set of irreducible monomial characters of $G$. We define an element $g\in G$ to be \emph{monomial-non-vanishing} if $\chi(g)\neq 0$ for every $\chi\in\irr_m(G)$ and denote by $\Nm(G)$ the set of such elements. Notice that $\N(G)\subseteq \Nm(G)$ while there are solvable groups for which the inclusion is strict (the smallest such group is $\SL_2(3)$). The Isaacs-Navarro-Wolf conjecture introduced in \cite{Isa-Nav-Wol} states that $\N(G)\subseteq\F(G)$ for every solvable group $G$. In this paper, we prove that $\Nm(G)\subseteq \F(G)$ in all the cases in which Isaacs--Navarro--Wolf's conjecture is known to hold. We conjecture that this fact holds for every solvable group.

\begin{conjA}
\label{conj:m-INW}
If $G$ is a solvable group, then $\Nm(G)\subseteq \F(G)$.
\end{conjA}

In \cite{Glu}, Gluck proved that the index of the Fitting subgroup of an arbitrary finite group $G$ is bounded by a polynomial function of the largest character degree $b(G)$ of $G$. In particular it is conjectured that $|G:\F(G)|\leq b(G)^2$ for every solvable group $G$. We define $b_m(G):=\{\chi(1)\mid \chi\in\irr_m(G)\}$ and observe that $b_m(G)\leq b(G)$. Moreover $b(G)-b_m(G)$ can be arbitrarily large (see Example \ref{ex:Gluck}). In Section \ref{sec:Gluck}, we verify the stronger bound $|G:\F(G)|\leq b_m(G)^2$ in all the cases in which Gluck's conjecture is known to hold. As before, we conjecture that this fact holds for all solvable groups.

\begin{conjA}
\label{conj:m-Gluck}
If $G$ is a solvable group, then $|G:\F(G)|\leq b_m(G)^2$.
\end{conjA}
 
To conclude, we show that the characterization of the normality of a Sylow $p$-subgroup given in \cite{Mal-Nav} can be checked by considering only monomial characters in the case of solvable groups.

\begin{theoA}
\label{thm:m-Malle-Navarro}
Let $G$ be a solvable groups, $p$ a prime and consider $P\in\syl_p(G)$. Then $P\unlhd G$ if and only if $p$ does not divide $\chi(1)$, for every $\chi\in\irr_m(G\mid 1_P)$.
\end{theoA}

The above results might lead to the following conceptual question for solvable groups: either monomial characters play an important role in the representation theory of solvable groups or they possess a large number of such characters. To support the first statement, consider $G:=\SL_2(3)$ and define $G_n$ to be the direct product of $n$ copies of $G$. Noticing that
\[\dfrac{|\irr_m(G_n)|}{|\irr(G_n)|}=\left(\dfrac{|\irr_m(G)|}{|\irr(G)|}\right)^n\]
by \cite{Wal}, we deduce that
$$\lim\limits_{n\to\infty}\frac{|\irr_m(G_n)|}{|\irr(G_n)|}=0$$
since $G$ is a non-monomial group. This shows that there exist solvable groups with an arbitrarily small proportion of monomial characters.

\section{Isaacs-Navarro-Wolf's conjecture}
\label{sec:INW}

In \cite{Isa-Nav-Wol}, Isaacs--Navarro--Wolf's conjecture is proved for groups of odd order and, more generally, for solvable groups with abelian Sylow $2$-subgroup. In \cite{Mor-Wol} the authors show that $\mathcal{N}(G)\subseteq \F_{10}(G)$ for every solvable group and extend the previous result to every solvable group $G$ in which every Fitting factor $\F_{i+1}(G)/\F_i(G)$ has abelian Sylow $2$-subgroup for every $1\leq i\leq 9$. The result of Moreto and Wolf is a consequence of a theorem on orbits of linear groups \cite[Theorem E]{Mor-Wol}. This has been improved by Yang \cite[Theorem 3.5]{Yan} with consequences on non-vanishing elements. As a consequence, Isaacs--Navarro--Wolf's conjecture holds for solvable groups in which every Fitting factor $\F_{i+1}(G)/\F_i(G)$ has abelian Sylow $2$-subgroup, for every $1\leq i\leq 7$. We are going to show that Conjecture \ref{conj:m-INW} holds under the same hypothesis. We start with the following preliminary result which should be compared with \cite[Lemma 2.3]{Isa-Nav-Wol}.

\begin{lem}
\label{lem:Fixed point}
Let $N\unlhd G$ and $x\in\mathcal{N}_m(G)$. If $N$ is abelian and $N\cap \Phi(G)=1$, then $x$ fixes a character in each $G$-orbit on $\irr(N)$.
\end{lem}

\begin{proof}
Fix $\lambda\in \irr(N)$. By hypothesis $N$ has a complement in $T:=G_\lambda$ \cite[III.4.4]{Hup67}. Thus $\lambda$ extends to $\wh{\lambda}\in\irr(T)$ by \cite[19.12]{Hup98} and $\chi:=\wh{\lambda}^G\in \irr_m(G)$ by the Clifford correspondence. Since $x\in\mathcal{N}_m(G)$, there exists $g\in G$ such that $gxg^{-1}\in T$. In particular $\lambda^g$ is fixed by $x$.
\end{proof}

Recall that a chief factor $K/L$ of $G$ is said to be \emph{non-Frattini} if $K/L\cap \Phi(G/L)=1$. In this case, if $K/L$ is abelian, then $K/L$ has a complement in $G/L$.

\begin{prop}
\label{prop:m-INW conjecture for odd order groups}
Let $G$ be solvable and let $x\in \mathcal{N}_m(G)$. Then the order of $\F(G)x$ in $G/\F(G)$ is a power of $2$. In particular Conjecture \ref{conj:m-INW} holds for groups of odd order. 
\end{prop}

\begin{proof}
Let $u$ be the $2'$-part of $x$. We show that $u\in\F(G)$. By \cite[13.8]{Doe-Haw}, it is enough to fix a non-Frattini chief factor $K/L$ of $G$ and prove that $u\in \mathbf{C}_G(K/L)=:C$. Since $1\neq \F(G)\leq C$, induction on $|G|$ yields $\overline{u}\in\F\left(\overline{G}\right)$, where $\overline{G}:=G/C$. Furthermore $\overline{G}$ acts faithfully and irreducibly on $\irr(K/L)$. Consider the quotient $G/L$ and observe that $Lx\in \mathcal{N}_m(G/L)$ and that $K/L$ satisfies the hypothesis of Lemma \ref{lem:Fixed point} in $G/L$. It follows that $Lx$ fixes a point in every $(G/L)$-orbit on $\irr(K/L)$, and so that $\overline{x}$ fixes a point in every $\overline{G}$-orbit on $\irr(K/L)$. In particular $\overline{u}$ fixes a point in each $\overline{G}$-orbit on $\irr(K/L)$ and, by \cite[Theorem 4.2]{Isa-Nav-Wol}, we conclude that $\overline{u}^2=1$. This shows that $u\in\F(G)$
\end{proof}

Next, we need analogues of \cite[Theorem 4.4]{Isa-Nav-Wol} and of \cite[Theorem 5.2]{Yan} for monomial characters.

\begin{prop}
\label{prop:Nonvanishing nonFitting}
Let $G$ be a solvable group and assume there exists some $x\in\mathcal{N}_m(G)\setminus\F(G)$. Then there exists a subgroup $\F(G)\leq N\unlhd G$ such that, if $\overline{G}:=G/N$, then $\overline{x}$ is an involution of $\F(\overline{G})$ and $\overline{x}\nin \overline{A}$, for every abelian normal subgroup $\overline{A}\unlhd \overline{G}$. 
\end{prop}

\begin{proof}
Let $M$ be a maximal element of the non-empty set $\{ N\unlhd G\mid Nx\nin\F(G/N)\}$. Replacing $G$ with $G/M$, we may assume that $Nx\in\F(G/N)$, for every $1<N\unlhd G$. Since $x\nin\F(G)$, by \cite[13.8]{Doe-Haw}, there exists a non-Frattini chief factor $K/L$ of $G$ such that $x\nin C:=\mathbf{C}_G(K/L)$. We claim that $C$ is the normal subgroup we are looking for. Let $\overline{G}:=G/C$ and observe that $1\neq \overline{x}\in\F(\overline{G})$ and that $\overline{G}$ acts faithfully and irreducibly on $\irr(K/L)$. By Lemma \ref{lem:Fixed point} we deduce that $Lx$ fixes a point in each $(G/L)$-orbit on $\irr(K/L)$. It follows that $\overline{x}$ fixes a point in each $\overline{G}$-orbit on $\irr(K/L)$ and we conclude by \cite[Theorem 4.2]{Isa-Nav-Wol}.
\end{proof}

\begin{prop}
\label{prop:Fitting 8}
Let $G$ be a solvable group. Then there exists $\mu\in\irr_m(\F_8(G))$ such that $\mu^G\in\irr_m(G)$. In particular $\Nm(G)\subseteq F_8(G)$.
\end{prop}

\begin{proof}
Since $\F_i(G/\Phi(G))=\F_i(G)/\Phi(G)$ for any $i\geq 1$, proceeding by induction on the $|G|$, we may assume $\Phi(G)=1$. Set $F:=\F_8(G)$. Notice that $\irr(\F(G))$ is a completely reducible and faithful $(G/\F(G))$-module and that, by \cite[3.5]{Yan}, there exists $\lambda\in\irr(\F(G))$ such that $T:=G_\lambda\leq F$. Using \cite[III.4.4]{Hup67} and \cite[19.12]{Hup98}, we deduce that $\lambda$ extends to a linear character $\wh{\lambda}\in\irr(T\mid \lambda)$. By the Clifford correspondence $\mu:=\wh{\lambda}^F$ has the required properties.
\end{proof}

As a consequence we obtain the result mentioned at the beginning of the section.

\begin{cor}
\label{cor:m-INW conjecture for abelian Sylow 2-subgroup of Fitting factors}
Conjecture \ref{conj:m-INW} holds for every solvable group $G$ in which $\F_{i+1}(G)/\F_i(G)$ has an abelian Sylow $2$-subgroup, for every $1\leq i\leq 7$.
\end{cor}

\begin{proof}
Let $x\in\mathcal{N}_m(G)$ and suppose that $x\nin\F(G)$. For $i\geq 1$, set $F_i:=\F_i(G)$ and notice that, by Proposition \ref{prop:Fitting 8}, there exists $1\leq i\leq 7$ such that $x\in F_{i+1}\setminus F_i$. Replacing $G$ with $G/F_{i-1}$, we may assume that $x\in F_2\setminus F_1$ and that $F_2/F_1$ has abelian Sylow $2$-subgroups. Let $F_1\leq N\unlhd G$ be as in Proposition \ref{prop:Nonvanishing nonFitting} and set $\overline{G}:=G/N$. Then $\overline{x}$ lies in a Sylow $2$-subgroup of $\overline{F}_2$ and, since such a subgroup is abelian and normal in $\overline{G}$, we have a contradiction.
\end{proof}

We end this section by mentioning two other cases in which Conjecture \ref{conj:m-INW} holds. First, suppose that $G/\F(G)$ is supersolvable. Under this hypothesis Issacs-Navarro-Wolf's conjecture holds by the results of \cite{He}.

\begin{prop}
\label{prop:m-INW for nilpotent-by-supersolvable}
Conjecture \ref{conj:m-INW} holds whenever $G/\F(G)$ is supersolvable.
\end{prop}

\begin{proof}
By induction on $|G|$ we may assume $\Phi(G)=1$. Now, $F:=\F(G)$ is abelian and has a supersolvable complement in $G$. By \cite[Theorem B]{Isa-Nav-Wol} and \cite[24.3]{Hup98}, we deduce that $\Nm(G/F)\subseteq \mathbf{Z}(\F(G/F))=:Z/F$ and that $\mathcal{N}_m(G)\subseteq Z$. Let $A:=H\cap Z$ and observe that $Z=F\rtimes A$. The abelian group $A$ acts faithfully on the completely reducible $A$-module $\irr(F)$. Noticing that, in this situation, complete reducibility coincide with the condition $(|G|,|\irr(F)|)$, we can find $\lambda\in\irr(F)$ such that $A_\lambda=1$ \cite[Lemma 3.1]{Isa-Nav-Wol}. In particular $Z_\lambda=F$ and $\vartheta:=\lambda^Z\in\irr(Z)$. Since $\vartheta$ vanishes on $Z\setminus F$, and so do all its conjugates, we conclude that every character $\chi\in\irr(G\mid \vartheta)$ vanishes on $Z\setminus F$. On the other hand, observe that $\lambda$ extends to $\wh{\lambda}\in\irr(G_\lambda)$ and that $\chi:=\wh{\lambda}^G\in\irr_m(G)$. Since $\chi$ lies over $\vartheta$, it follows that $\chi$ vanishes on $Z\setminus F$ and so $\Nm(G)\subseteq F$.
\end{proof}

Finally, assume that $4$ does not divide the degree of any irreducible monomial character.

\begin{lem}
\label{lem:m-Le-Na-Wo}
Let $G$ be a solvable group and suppose that $p^2$ does not divide $\chi(1)$, for all $\chi\in\irr_m(G)$. Then the Sylow $p$-subgroup of $\F_{i+1}(G)/\F_i(G)$ has order $1$ or $p$, for every $i\geq 1$.
\end{lem}

\begin{proof}
First, notice that it's enough to show the result for the Sylow $p$-subgroup $S/F$ of $F_2/F$, where $F:=\F(G)$ and $F_2:=\F_2(G)$. Furthermore, since $\F_i(G/\Phi(G))=\F_i(G)/\Phi(G)$, we may assume $\Phi(G)=1$. Now, $G$ splits over the abelian normal subgroup $F$ and every $\lambda\in\irr(F)$ has an extension $\wh{\lambda}\in\irr(G_\lambda)$. By the Clifford correspondence $\wh{\lambda}^G\in\irr_m(G)$ and so $p^2$ does not divide $|G:G_\lambda|$. In particular $|S:S_\lambda|$ is either $1$ or $p$. Next observe that, since $\irr(F)$ is a completely reducible and faithful $(S/F)$-module, by \cite[Theorem A]{Wol} there exists $\lambda_1\in\irr(F)$ such that $|S:F|^{1/2}\leq |S:S_{\lambda_1}|$. In particular $S/F$ is abelian. This, together with \cite[Lemma 3.1]{Isa-Nav-Wol}, implies that there exists $\lambda_2\in \irr(F)$ such that $F=S_{\lambda_2}$ and we conclude by the previous discussion.
\end{proof}

\begin{cor}
Conjecture \ref{conj:m-INW} holds whenever $G$ is a solvable group such that $4$ does not divide $\chi(1)$, for all $\chi\in\irr_m(G)$.
\end{cor}

\begin{proof}
This follows by Corollary \ref{cor:m-INW conjecture for abelian Sylow 2-subgroup of Fitting factors} and Lemma \ref{lem:m-Le-Na-Wo}. 
\end{proof}

Notice that Lemma \ref{lem:m-Le-Na-Wo} is an adaptation of \cite[Lemma 2.1]{Lew-Nav-Wol}. One may wonder if the main result of that paper holds by considering only monomial characters.

\begin{pb}
\label{pb:Fitting index}
Let $G$ be a solvable group and suppose that $p^2$ does not divide $\chi(1)$, for all $\chi\in\irr_m(G)$. Is it true that $|G:\F(G)|_p\leq p^2$?
\end{pb}

To end this section we mention that, by introducing the ideas used in this section in the setting of Brauer non-vanishing elements, one can prove a monomial version of the main result of \cite{Dol-Pac-San}.

\section{Gluck's conjecture}
\label{sec:Gluck}

Gluck's conjecture has been proved by Espuelas for groups of odd order \cite[Theorem 3.2]{Esp}, by Dolfi and Jabara for solvable groups with abelian Sylow $2$-subgroups \cite[Theorem 1]{Dol-Jab} and by Yang for solvable groups in which $3$ does not divide $|G:\mathbf{F}(G)|$ \cite[Theorem 2.5]{Yan-Gluck}. Other partial results are obtained in \cite{Cos}. Moreover, in \cite[Corollary 2.7]{Mor-Wol} the authors prove that $|G:\mathbf{F}(G)|\leq b^3(G)$ for every solvable group $G$. We show that Conjecture \ref{conj:m-Gluck} holds in all the above mentioned situations. First, we collect some results on actions of solvable linear groups. For every integer $n$, we denote by $\pi(n)$ the set of prime divisors of $n$.

\begin{theo}
\label{thm:Large orbits}
Let $V$ be a finite faithful completely reducible $G$-module and $\pi=\pi(V)$. Define $\gamma$ as follows:
\begin{enumerate}
\item Set $\gamma:=2/3$ if $G$ is $\pi$-solvable.
\item Set $\gamma:=1/2$ if one of the following holds:
\begin{itemize}
\item $G$ is solvable and $V\rtimes G$ has abelian Sylow $2$-subgroup.
\item $G$ is a solvable $3'$-group. 
\item $G$ is solvable and $V$ is primitive with either $|V|\neq 3^4$ or $|V|=3^4$ and $G$ not conjugate to a subgroup of $\emph{GL}(V)$ of order $1152$.
\item $G$ is $\pi$-solvable and $64\cdot 81$ doesn't divide $|V|$.
\end{itemize}
\end{enumerate}
Then there exists $v\in V$ such that $|\mathbf{C}_G(v)|\leq |G|^\gamma$.
\end{theo}

\begin{proof}
See \cite[Lemma 11]{Cos}, \cite[Theorem 2]{Dol-Jab}, \cite[Theorem 3.1]{Esp}, \cite[Corollary 2.6]{Mor-Wol} and \cite[Corollary 2.4]{Yan-Gluck}.
\end{proof}

\begin{prop}
\label{prop:Monomial Gluck's conjecture}
Let $G$ be a group and set $V:=\irr(\mathbf{F}^*(G/\Phi(G)))$. Let $\pi:=\pi(\mathbf{F}^*(G/\Phi(G)))$, consider $\gamma$ as in Theorem \ref{thm:Large orbits} and define $\alpha:=1/(1-\gamma)$. Then 
$$|G:\mathbf{F}(G)|\leq b_m(G)^{\alpha}$$
where $b_m(G)$ is the largest irreducible monomial character degree of $G$.
\end{prop}

\begin{proof}
Since, by hypothesis, $\mathbf{F}^*(G/\Phi(G))=\mathbf{F}(G/\Phi(G))=\F(G)/\Phi(G)$, we may assume $\Phi(G)=1$. Now, $F:=\mathbf{F}(G)$ is abelian and $V=\irr(F)$ is a finite faithful completely reducible $(G/F)$-module. By Theorem \ref{thm:Large orbits} there exists $\lambda\in V$ such that $|G_\lambda:F|\leq |G:F|^\gamma$. Let $\wh{\lambda}\in\irr(G_\lambda)$ be an extension of $\lambda$ and consider $\chi:=\wh{\lambda}^G\in \irr_m(G)$. Then, $|G:F|=|G:F|^\alpha|G:F|^{-\alpha\gamma}\leq |G:F|^\alpha|G_\lambda:F|^{-\alpha}=|G:G_\lambda|^\alpha=\chi(1)^\alpha\leq b_m(G)^\alpha$.
\end{proof}

We remark that Conjecture \ref{conj:m-Gluck} holds also when $|V|=3^4$ and $G$ is conjugate to a subgroup of $\text{GL}(V)$ of order $1152$. In this case $\GL(V)$ has three conjugacy classes of subgroups of order $1152$ and the subgroups of one of these classes do not have orbits of size at least $|1152|^{1/2}$. However, using GAP \cite{Gap} one can check that the result holds for this subgroups as well.

As we have already mentioned, there exist solvable groups in which $b(G)$ is arbitrarily larger than $b_m(G)$.

\begin{ex}
\label{ex:Gluck}
Let $q$ and $p$ be odd primes such that $p\equiv -1\pmod{q}$ and consider $P$ an extraspecial $p$-group of order $p^3$ and exponent $p$.  Let $C\leq \text{Aut}(P)$ be a cyclic group of order $q$ acting trivially on $\z(P)$ and define $G:=P\rtimes C$. By elementary character theory, we deduce that the only irreducible character degrees of $G$ are $1$, $q$ and $p$ and that no monomial characters can have degree $p$. This shows that $b_m(G)=q<p=b(G)$. Notice that the prime $p$ can be chosen to be arbitrarily large by Dirichlet's theorem on arithmetic progressions. 
\end{ex}

\section{A characterization of normal Sylow $p$-subgroups}
\label{sec:Sylow}

In this final section, we consider the characterization given by Malle and Navarro in \cite{Mal-Nav} and obtain Theorem \ref{thm:m-Malle-Navarro}. We prove a slightly more general result which also implies \cite[Theorem 1.3]{Lu-PanI} and \cite[Theorem 1.1 (2)]{Lu-PanII}.

\begin{theo}
\label{thm:m-Malle-Navarro with kernels}
Let $G$ be a finite group and set
$$G_m^*(p):=\bigcap\limits_\chi \ker(\chi),$$
where $\chi$ runs over all characters in $\irr_m(G\mid 1_P)$ which vanish on some $p$-element of $G$. If $G^*_m(p)$ is solvable, then it has a normal Sylow $p$-subgroup.
\end{theo}

\begin{proof}
Let $G$ be a minimal counterexample. Set $L:=G^*_m(p)$ and notice that $L\unlhd G$. Let $Q$ be a Sylow $p$-subgroup of $L$. If $M$ is a minimal normal subgroup of $G$ with $M\leq L$, then $L/M=(G/M)^*_m(p)$ and we obtain $QM\unlhd L$, which implies $QM\unlhd G$. In particular $\O_p(L)=1$ and, as $L$ is solvable, we deduce that $M$ is an abelian $p'$-subgroup. Furthermore, by the Frattini argument, $G=M\rtimes N$ and $L=M\rtimes (N\cap L)$, where $N:=\n_G(Q)$. Consider $M<K\leq MQ$ such that $K/M$ is a chief factor of $G$ and set $H:=K\cap Q$. Observe that $H\neq 1$ and that $\mathbf{C}_H(M)\leq \mathbf{O}_p(L)=1$. Therefore $H$ acts faithfully on $\irr(M)$ and, by \cite[Lemma 3.1]{Isa-Nav-Wol}, there exists $\lambda\in \irr(M)$ such that $H_\lambda=1$. As a consequence $K_\lambda=M$ and $\vartheta:=\lambda^K$ is an irreducible character of $K$ vanishing on the normal subset $K\setminus M$. By Clifford's theorem, every character $\chi\in\irr(G\mid \vartheta)$ vanishes on $K\setminus M$ and, in particular, on $H\setminus 1$.

Recall that $G$ splits over $M$ and let $\wh{\lambda}$ be an extension of $\lambda$ to $T:=G_\lambda$. Consider $\alpha_0:=\wh{\lambda}_{T\cap N}$ and let $\alpha\in \irr(T/M)$ be the character corresponding to $\alpha_0$ via the isomorphism $T/M\simeq T\cap N$. Then, $\beta:=\wh{\lambda}\alpha^{-1}$ is a linear character of $T$ lying over $\lambda$ and $\chi:=\beta^G\in\irr_m(G)$ by the Clifford correspondence. Notice that $\chi$ lies over $1_P$ by the MacKay formula. On the other hand, observe that $\chi$ lies over $\beta$, hence over $\lambda$ and so over $\vartheta=\lambda^K$. By the previous paragraph, we deduce that $L=G^*_m(p)\leq \ker(\chi)$. This implies that $M\leq \ker(\vartheta)$ and so $\vartheta\in\irr(K/M)$. Since $K/M$ is abelian, we conclude that $1=\vartheta(1)=|K:M|$ a contradiction.
\end{proof}

As an immediate consequence, we obtain the following strong form of Theorem \ref{thm:m-Malle-Navarro}.

\begin{cor}
Let $G$ be a solvable groups, $p$ a prime and consider $P\in\syl_p(G)$. Then the following are equivalent:
\begin{enumerate}
\item $P\unlhd G$;
\item $p$ does not divide $\chi(1)$, for every $\chi\in\irr_m(G\mid 1_P)$;
\item $\chi(x)\neq 0$, for every $\chi\in\irr_m(G\mid 1_P)$ and $x\in P$.
\end{enumerate}
\end{cor}

\begin{proof}
First notice that (i) implies (ii) and that (ii) implies (iii). Then we conclude by Theorem \ref{thm:m-Malle-Navarro with kernels}.
\end{proof}

The above result does not hold for arbitrary finite groups. For instance, consider $G=A_5$, $p=2$ and observe that the irreducible monomial characters of $G$ lying over the principal character of a Sylow $p$-subgroup are the trivial character and the character of degree $5$.

\bibliographystyle{alpha}
\bibliography{References}
\vspace{1cm}

DEPARTMENT OF MATHEMATICS, CITY, UNIVERSITY OF LONDON, EC$1$V $0$HB, UNITED KINGDOM.
\textit{Email address:} \href{mailto:damiano.rossi@city.ac.uk}{damiano.rossi@city.ac.uk}
\end{document}